\documentclass[final]{dmtcs-episciences}
\usepackage{amsmath,amsfonts,amssymb,amsthm,mathtools}
\usepackage{microtype}
\usepackage{lipsum}
\usepackage{enumerate}
\usepackage{graphicx}
\usepackage[margin=1.25in]{geometry}
\usepackage{color}
\usepackage{comment}
\usepackage{hyperref}
\usepackage{cleveref}
\usepackage{mathdots}
\usepackage[utf8]{inputenc}
\allowdisplaybreaks

\newcommand\abs[1]{\left\lvert#1\right\rvert}
\newcommand\ceil[1]{\left\lceil#1\right\rceil}
\newcommand\floor[1]{\left\lfloor#1\right\rfloor}
\DeclareMathOperator\forb{forb}
\DeclareMathOperator\Avoid{Avoid}
\DeclareMathOperator\ext{ext}
\makeatletter
\renewcommand*\env@matrix[1][*\c@MaxMatrixCols c]{%
  \hskip -\arraycolsep
  \let\@ifnextchar\new@ifnextchar
  \array{#1}}
\makeatother

\newtheorem{thm}{Theorem}[section]
\newtheorem{lemma}[thm]{Lemma}
\newtheorem{prop}[thm]{Proposition}
\newtheorem{cor}[thm]{Corollary}
\newtheorem{definition}[thm]{Definition}
\newtheorem{remark}[thm]{Remark}
\newcommand{\ncols}[1]{| #1 |}

\usepackage[round]{natbib}

\title{Multivalued forbidden numbers of two-rowed configurations -- the missing cases}
\date{}

\author{Wallace Peaslee\affiliationmark{1}\thanks{Supported by a UK Engineering and Physical Sciences Research Council Doctoral Training Partnership grant.}  
\and Attila Sali\affiliationmark{2,3}\thanks{Research was partially supported by the National Research, Development and Innovation Office (NKFIH) grants K--132696 and SNN-135643. }   
\and Jun Yan\affiliationmark{4,5}\thanks{Was supported by the Warwick Mathematics Institute Centre for Doctoral Training and by funding from the UK EPSRC (Grant number: EP/W523793/1) during this project. }  }
\affiliation{
University of Cambridge, Department of Applied Mathematics and Theoretical Physics, Cambridge, United Kingdom\\
HUN-REN Alfr\'ed R\'enyi Institute of Mathematics, Budapest. Hungary\\
Department of Computer Science, BUTE, Budapest, Hungary\\
University of Warwick, Warwick, England\\ 
Faculty of Mathematics, University of Waterloo, Waterloo, Canada
}
\begin{document}
\publicationdata {vol. 28:3}{2026}{17}{10.46298/dmtcs.15215}{2025-02-10; 2025-02-10; 2026-09-29}{2026-10-01}
\maketitle{}
\begin{abstract}
    The present paper considers extremal combinatorics questions in the language of matrices. An $s$-matrix is a matrix with entries in $\{0,1,\ldots, s-1\}$. An $s$-matrix is simple if it has no repeated columns. A matrix $F$ is a configuration in a matrix $A$, denoted $F\prec A$, if it is a row/column permutation of a submatrix of $A$. $\Avoid(m,s,F)$ is the set of $m$-rowed, simple $s$-matrices not containing a configuration of $F$ and $\forb(m,s, F)=\max\{|A|\colon A \in \Avoid(m,s,F)\}$. Dillon and Sali initiated the systematic study of $\forb(m,s, F)$ for $2$-matrices $F$, and computed $\forb(m,s, F)$ for all 2-rowed $F$ when $s>3$. In this paper we tackle the remaining cases when $s=3$. In particular, we determine the asymptotics of $\forb(m,3,p\cdot K_2)-\forb(m,3,p\cdot I_2)$ for $p>3$, where $K_2$ is the $2\times 4$ simple $2$-matrix and $I_2$ is the $2\times 2$ identity matrix, as well as the exact values of $\forb(m,3,F)$ for many 2-rowed $2$-matrices $F$.
    
\end{abstract}

\section{Introduction}
Several questions of extremal combinatorics can be conveniently expressed in the language of $(0,1)$-matrices. A matrix $A$ is said to be \emph{simple} if it contains no repeated columns. There is a natural correspondence between columns of an $m$-rowed $(0,1)$-matrix and subsets of $[m]$. We consider an extremal set theory problem in matrix terminology as follows. Let $\ncols{A}$ be the number of columns of $A$. For a given matrix $F$, we say $F$ is a \emph{configuration} in $A$, denoted $F\prec A$, if there is a submatrix of $A$ which can be obtained by permuting the rows and columns of $F$. 
Define 
\begin{displaymath}
    \Avoid(m,F)=\left\{A\colon A \hbox{ is an }m\hbox{-rowed }(0,1)\hbox{-matrix and } F\not\prec A\right\},\end{displaymath}
\begin{displaymath}\forb(m,F)=\max\{\ncols{A}:A\in\Avoid(m,F)\}.\end{displaymath}
We say a matrix $A\in\Avoid(m,F)$ is {\emph{extremal}} if $\ncols{A}=\forb(m,F)$, and let 
\begin{displaymath}\ext(m,F)=\{A\in\Avoid(m,F)\colon\ncols{A}=\forb(m,F)\}.\end{displaymath} 

The foremost theorem in forbidden configurations is Theorem~\ref{thm:sauer} below. Here, $K_k$ denotes the $k\times 2^k$ matrix containing all possible $0,1$-columns.
\begin{thm}\label{thm:sauer}\citet{sauer,shelah,vapnik}
    \begin{displaymath}
    \forb(m,K_k)=\binom{m}{k-1}+\binom{m}{k-2}+\cdots+\binom{m}{0}.
    \end{displaymath}
\end{thm}
Further interesting results and research problems can be found in the survey paper \citet{survey}. However, our main interest in the present paper is the generalization of these concepts to $s$-matrices, which have entries in $\{0,1,\ldots, s-1\}$. Such matrices can be thought of as $s$-coloured set systems or as representations of collections of functions from a given finite set into $\{0,1,\dots,s-1\}$.
The concepts of \emph{simple} matrices and \emph{configurations} naturally extend. For a given collection $\mathcal F$ of matrices, we denote by $\Avoid(m,s,{\mathcal F})$ the collection of $m$-rowed, simple $s$-matrices that avoid every matrix $F \in \mathcal F$. The main extremal function in the study of forbidden configurations is
\begin{displaymath}
    \forb(m,s,{\mathcal F})=\max\{|A|\colon A \in \Avoid(m,s,{\mathcal F})\},
\end{displaymath}
which we call the \emph{forbidden number}, and we similarly define 
\begin{displaymath}\ext(m,s,F)=\{A\in\Avoid(m,s,F)\colon\ncols{A}=\forb(m,s,F)\}.\end{displaymath}

Alon \citet{alon1983} and Steele \citet{steele1978existence} gave a generalization of Theorem~\ref{thm:sauer} to 
$s$-matrices, with the bound given being an exponential function of $m$. This is a special case of the following more general phenomenon proved by F\"uredi and Sali \citet{FS}. A matrix is called an \emph{$(i,j)$-matrix} if its entries are all from the set $\{i,j\}$.

\begin{thm}
Let ${\mathcal F}$ be a family of $s$-matrices. If for every pair of distinct $i,j\in\{0,1,\dots,s-1\}$ there is an $(i,j)$-matrix in ${\mathcal F}$, then $\forb(m,s,{\mathcal F}) = O(m^k)$ for some positive integer $k$. If ${\mathcal F}$ has no $(i,j)$-matrix for some distinct $i,j \in \{0,1,\dots,s-1\}$, then $\forb(m,s,{\mathcal F}) = \Omega(2^m)$.
\end{thm}
Several papers \citet{anstee2023multivalued}, \citet{ansteelu}, \citet{ELLIS202024} considered special sets of forbidden configurations for which the forbidden numbers are of polynomial magnitudes. The systematic investigation of exact exponential bounds when forbidding $(0,1)$-configurations was started in \citet{dillon2021exponential}. They provided exact formulas for the forbidden numbers $\forb(m,s,F)$ of many $(0,1)$-matrices $F$, including all $2$-rowed matrices when $s > 3$. Along the way, they exposed some interesting qualitative differences between the cases $s=2$, $s = 3$, and $s > 3$. 

Some important matrices include $I_k$, the $k\times k$ identity matrix  and $K_k$, the $k\times 2^k$ matrix of all possible (0,1)-columns on $k$ rows. We use the following notation for general 2-rowed configurations: \begin{displaymath}F(a,b,c,d)=\left[\begin{array}{@{}c@{}}\\ \\ \end{array}\right. \overbrace{
\begin{array}{@{}c}00\cdots 0\\ 00\cdots 0\\ \end{array}}^a \overbrace{
\begin{array}{c}11\cdots 1\\ 00\cdots 0\\ \end{array}}^b \overbrace{
\begin{array}{c}00\cdots 0\\ 11\cdots 1\\\end{array}}^c 
\overbrace{
\begin{array}{c@{}}11\cdots 1\\ 11\cdots 1\\\end{array}}^d \left.\begin{array}{@{}c@{}}\\ \\ \end{array}\right].\end{displaymath}
For example, $F(1,1,1,1)=K_2$. Let $p\cdot F$ denote the concatenation of $p$ copies of $F$, so that in particular $p\cdot K_2$ and $F(p,p,p,p)$ are the same as configurations, while $p\cdot I_2$ is the same as $F(0,p,p,0)$.

When $s>3$, the forbidden numbers $\forb(m,s,F(a,b,c,d))$ for all 2-rowed configurations were determined exactly in \citet{dillon2021exponential}. Under mild conditions similar to Lemma \ref{reduction} below, it was shown that $\forb(m,s,F(a,b,c,d))=\forb(m,s,F(a,p,p,d))$, where $p=\min\{b,c\}$. In particular, it turned out that $\forb(m,s,p\cdot K_2)=\forb(m,s,p\cdot I_2)$, which implies that $\forb(m,s,p\cdot K_2)=\forb(m,s,F(a,p,p,d))$ if $s>3$ and $a,d\le p$.

For general 2-rowed configurations $F(a,b,c,d)$, the case when $s=3$ is harder than when $s>3$. In particular, it was shown in \citet{dillon2021exponential} that $\forb(m,3,p\cdot K_2)>\forb(m,3,p\cdot I_2)$ if $p\ge 3$, exhibiting a different behaviour to the $s>3$ case. More generally, $\forb(m,3,F(a,b,c,d))$ were determined in \citet{dillon2021exponential} only when $\max\{a,d\}\ge\min\{b,c\}$. Our goal in this paper is to study the missing cases when $\max\{a,d\}<\min\{b,c\}$. 

The following important reduction lemma was proved in \citet{dillon2021exponential}.
\begin{lemma}\citet{dillon2021exponential}\label{reduction}
Let $p=\min\{b,c\}$. If $2^{m-2}\geq(\max\{a,b,c,d\}-1)m^2$ and $p\geq1$, then
\begin{displaymath}\forb(m,3,F(a,b,c,d))=\forb(m,3,F(a,p,p,d)).\end{displaymath}
\end{lemma}

Note that by Lemma~\ref{reduction}, we may assume without loss of generality that $b=c=p$ provided that $m$ is large. It was also proved in \citet{dillon2021exponential} that $\forb(m,3,F(p,p,p,p))=\forb(m,3,p\cdot K_2)=2^m+m2^{m-1}+(p-1)\binom{m}{2}$. In what follows, we investigate the forbidden numbers $\forb(m,3,F(a,p,p,d))$ for various parameters $a,d<p$.

In Section~\ref{sec:fabcd}, we begin by introducing some terminologies and notations. Then, we prove first a generalization of a lemma in \citet{dillon2021exponential}, which roughly says that any matrix in $\Avoid(m,3,F(a,p,p,d))$ that contains many columns must satisfy some stability-type structural properties. In Section~\ref{sec:gp}, we determine the difference between $\forb(m,3,p\cdot K_2)$ and $\forb(m,3,p\cdot I_2)$ asymptotically. Throughout the paper we let $r=\ceil{\log_2(p-1)}$, so $2^{r-1}<p-1\leq 2^r$. Our first, maybe surprising, observation is that the difference above does not depend on $m$.
\begin{prop}\label{prop:gpindep}
For all $p\geq 2$, $\forb(m,3,p\cdot K_2)-\forb(m,3,p\cdot I_2)$ is independent of $m$ as long as $m\geq 2r+2$.
\end{prop}
This allows us to define $g_p$ to be the common value of $\forb(m,3,p\cdot K_2)-\forb(m,3,p\cdot I_2)$ for $m\geq 2r+2$, and we prove the following asymptotic result for $g_p$.  
\begin{thm}\label{thm:gpasymp}
\begin{displaymath}g_p\sim\frac12p(\log_2p)^2.\end{displaymath}
\end{thm}

In Section~\ref{subsec:p34}, we compute $\forb(m,3,F(a,p,p,d))$ for all $p\leq4$ and $\max\{a,d\}<p$, which all turn out to be equal to $\forb(m,3,p\cdot I_2)$.

This no longer holds for larger values of $p$, but in Section~\ref{sec:p-k} we have further exact results for many cases. Let \begin{displaymath}N_r=\binom{\ceil{\frac{r}{2}}+1}{2}+\binom{\floor{\frac{r}{2}}+1}{2}.\end{displaymath}
The next theorem determines $\forb(m,3,F(p-q,p,p,p-q))$ for a wide range of values $p$ and $q$.

\begin{thm}\label{thm:p-k}
For all $m\geq 2N_r+2$ and $0\leq q\leq p-1$, \begin{displaymath}\forb(m,3,F(p-q,p,p,p-q))=\forb(m,3,p \cdot K_2)-qN_r\end{displaymath}
if one of the following holds
\begin{itemize}
    \item $p\geq2^{r-1}+2q+1$,
    \item $p\leq2^{r-1}+2q$, $qN_r\leq2^{r-3}$ and $\frac12(\floor{\frac r2}+1)(p-1-2^{r-1})\geq q$. 
\end{itemize}
In particular, this is true for all fixed constants $q$ and all $p$ sufficiently large.
\end{thm}
With a more careful analysis of the proof of Theorem~\ref{thm:p-k}, we prove the following result.
\begin{cor}\label{thm:p-1ppp-1}
For all $m\geq 2N_r+2$ and $p\geq 2$, \begin{displaymath}\forb(m,3,F(p-1,p,p,p-1))=\forb(m,3,p \cdot K_2)-N_r.\end{displaymath}
\end{cor}

Determining the exact forbidden number in the remaining cases when $q$ is relatively large might be difficult, given that the exact value of $\forb(m,3,p\cdot I_2)$ is not known. In particular, the following result shows that the exact formula in Theorem \ref{thm:p-k} may not hold if those conditions are not satisfied. 
\begin{prop}\label{prop:notalwaystrue}
If $m\geq 2N_r+2$, $0\leq q\leq p-1$ and $d:=q-(\ceil{\frac{r}2}+1)(p-1-2^{r-1})>0$, then \begin{displaymath}\forb(m,3,F(p-q,p,p,p-q))\geq\forb(m,3,p \cdot K_2)-qN_r+\floor{\frac r2}d.\end{displaymath}
\end{prop}

We also have the following more general lower bound.
\begin{prop}\label{prop:lowerbound}
For every $0\leq q_0,q_1\leq p-1$, every $0\leq r_1,r_2\leq r$ satisfying $r_1+r_2=r$, and every $m\geq 2\binom{r_1+1}2+2\binom{r_2+1}2+2$, we have
\begin{displaymath}\forb(m,3,F(p-q_0,p,p,p-q_1))\geq\forb(m,3,p\cdot K_2)-\min\{q_0,q_1\}\binom{r_1+1}2-\max\{q_0,q_1\}\binom{r_2+1}2.\end{displaymath}
\end{prop}
For fixed $p,q_0,q_1$, we may optimize the right-hand side by viewing $r_1$ as the variable and using $r_2=r-r_1$. That is, assuming $q_0\leq q_1$, we minimize $q_0\binom{r_1+1}2+q_1\binom{r-r_1+1}2$. Easy calculation shows that if $q_1=\alpha q_0$ for some $\alpha\ge 1$, then the minimum is achieved when $r_1$ is the closest integer to $\frac{(2r+1)\alpha-1}{2(\alpha+1)}$, as long as this fraction is at most $r$, which is equivalent to $\alpha\leq 2r+1$. Otherwise, $r_1=r$ is optimal. Note that if $q_0=q_1$ and so $\alpha=1$, then the minimum is achieved when $r_1=\ceil{\frac r2}, r_2=\floor{\frac r2}$. The resulting lower bound in this case is tight in many cases by Theorem \ref{thm:p-k}. We suspect that a more involved version of this proof could potentially show that after optimization, the general lower bound for $\forb(m,3,F(p-q_0,p,p,p-q_1))$ given by Proposition \ref{prop:lowerbound} is also tight in several cases.

\section{Results}\label{sec:fabcd}
Recall that from Lemma \ref{reduction}, we may assume that $b=c$, as long as $m$ is sufficiently large. From now on, we will assume $b=c=p\geq 3$, $0\leq a,d\leq p-1$ unless stated otherwise. Generalizing some ideas from \citet{dillon2021exponential}, we now make some definitions to help better understand the structure of matrices in $\Avoid(m,3,F(a,p,p,d))$.

\begin{definition}\label{def:edge-nonedge}
Let $A\in\Avoid(m,3,F(a,p,p,d))$. To each pair $(i,j)$ with $1\leq i<j\leq m$, associate a vector $\begin{bmatrix}x\\y\end{bmatrix}\in\left\{\begin{bmatrix}0\\0\end{bmatrix}, \begin{bmatrix}0\\1\end{bmatrix}, \begin{bmatrix}1\\0\end{bmatrix}, \begin{bmatrix}1\\1\end{bmatrix}\right\}$ so that if $\begin{bmatrix}x\\y\end{bmatrix}=\begin{bmatrix}0\\1\end{bmatrix}$ or $\begin{bmatrix}1\\0\end{bmatrix}$ then rows $i,j$ contain at most $p-1$ copies of $\begin{bmatrix}x\\y\end{bmatrix}$, if $\begin{bmatrix}x\\y\end{bmatrix}=\begin{bmatrix}0\\0\end{bmatrix}$ then rows $i,j$ contain at most $a-1$ copies of $\begin{bmatrix}0\\0\end{bmatrix}$ and if $\begin{bmatrix}x\\y\end{bmatrix}=\begin{bmatrix}1\\1\end{bmatrix}$then rows $i,j$ contain at most $d-1$ copies of $\begin{bmatrix}1\\1\end{bmatrix}$. Such a vector must exist as $F(a,p,p,d)\nprec A$, and if multiple choices are possible, we pick one arbitrarily. In all cases, we say a column $v$ of $A$ has a mark at the pair $(i,j)$ if $v_i=x$ and $v_j=y$. 

Let $B$ be the matrix consisting of all columns of $A$ with no mark, and let $C$ be the matrix consisting of all columns of $A$ with at least one mark. We call $A=\begin{bmatrix}B\mid C\end{bmatrix}$ the standard decomposition of $A$. 

We also associate to $A$ a directed graph $T$ on $[m]$ as follows: for each pair $1\leq i<j\leq m$, if $\begin{bmatrix}0\\1\end{bmatrix}$ is associated with $(i,j)$, add the directed edge $ij$; if $\begin{bmatrix}1\\0\end{bmatrix}$ is associated with $(i,j)$, add the directed edge $ji$; otherwise do nothing. Let $N$ denote the number of pairs $(i,j)$ that are associated with either $\begin{bmatrix}0\\0\end{bmatrix}$ or $\begin{bmatrix}1\\1\end{bmatrix}$. We say that these pairs form 0-non-edges and 1-non-edges in $T$, respectively.
\end{definition}

Using the notations above, we have some immediate observations. It follows from the definition that $B$ is in $\Avoid(m,3,K_2)$ so $\abs{B}\leq 2^m+m2^{m-1}$. We also have $\abs{C}\leq(p-1)(\binom{m}{2}-N)+N(\max\{a,d\}-1)=(p-1)\binom{m}{2}-N(p-\max\{a,d\})$, since every column in $C$ contains at least one mark. 

The following key lemma, generalizing one from \citet{dillon2021exponential}, shows that if $|B|$ is large, then the associated directed graph $T$ is transitive.

\begin{lemma}\label{transitive}
If $\abs{B}>2^m+m2^{m-1}-2^{m-3}$, then $T$ is transitive. 
\end{lemma}
\begin{proof}
Assume $T$ is not transitive, then there exists a triple $(i,j,k)$ of distinct vertices such that $ij,jk$ are edges of $T$ but either $ki$ is an edge of $T$ or there is no edge between $k$ and $i$. For a column $v$ in $B$, define its \emph{support} to be the set of row indices $i$ such that $v_i$ is either 0 or 1. For any set $X\subset[m]$ of rows of $B$, let $B(X)$ be the matrix consisting of columns of $B$ with support $X$. Then its restriction $B(X)|_X$ of $B(X)$ to the rows in $X$ is simple and is in $\Avoid(\abs{X},2,K_2)$. Thus, if $\{i,j,k\}\not\subset X$, then $\abs{B(X)}=\abs{B(X)|_X}\leq\forb(|X|,2,K_2)=\abs{X}+1$ by Theorem \ref{thm:sauer}. 

If, however, $\{i,j,k\}\subset X$, then consider $B(X)|_{\{i,j,k\}}$. If $ij,jk,ki$ are all edges of $T$, the only possible columns in $B(X)|_{\{i,j,k\}}$ are $\begin{bmatrix}0\\0\\0\end{bmatrix}$ and $\begin{bmatrix}1\\1\\1\end{bmatrix}$. So $B(X)|_{X\setminus\{i,j\}}$ is simple and avoids $K_2$. Thus $\abs{B(X)}=\abs{B(X)|_{X\setminus\{i,j\}}}\leq\forb(|X|-2,2,K_2)=\abs{X}-1$. If $ij,jk$ are edges of $T$ but there is no edge between $i$ and $k$, there are two cases. If rows $i,k$ are associated with $\begin{bmatrix}0\\0\end{bmatrix}$, then any column in $B(X)|_{\{i,j,k\}}$ has a $1$ in entry $i$ so ${B(X)|_{X\setminus\{i\}}}$ is simple and avoids $K_2$. If rows $i,k$ are associated with $\begin{bmatrix}1\\1\end{bmatrix}$, then any column in $B(X)|_{\{i,j,k\}}$ has a $0$ in entry $k$ so ${B(X)|_{X\setminus\{k\}}}$ is simple and avoids $K_2$. In both cases, we have $\abs{B(X)}\leq\forb(|X|-1,2,K_2)=\abs{X}$.

\indent Therefore, if $ij,jk,ki$ are all edges of $T$, by counting columns in $B$ according to their support, we have
\begin{align*}
\abs{B}&\leq\sum_{\substack{X\subset[m]\\\{i,j,k\}\not\subset X}}(\abs{X}+1)+\sum_{\substack{X\subset[m]\\\{i,j,k\}\subset X}}(\abs{X}-1)\\
&=\sum_{X\subset[m]}(\abs{X}+1)-2\sum_{\substack{X\subset[m]\\\{i,j,k\}\subset X}}1\\
&=\sum_{j=0}^m\binom{m}{j}(j+1)-2\sum_{j=0}^{m-3}\binom{m-3}{j}=2^m+m2^{m-1}-2^{m-2}.    
\end{align*}
A similar calculation shows that $\abs{B}\leq2^m+m2^{m-1}-2^{m-3}$ if $ij,jk$ are edges of $T$ but neither $ik$ nor $ki$ is. Both contradict the assumption that $\abs{B}>2^m+m2^{m-1}-2^{m-3}$, which completes the proof.
\end{proof}

Note that if $\abs{A}>2^m+m2^{m-1}-2^{m-3}+(p-1)\binom{m}{2}$, then as $|C|\leq(p-1)\binom m2$, we must have $|B|>2^m+m2^{m-1}-2^{m-3}$. Therefore, by Lemma \ref{transitive}, $T$ is transitive, and hence the directed comparability graph of a finite poset $P$. Thus, by taking a linear extension of $P$, we can rearrange the rows of $A$ so that no pair $1\leq i<j\leq m$ is associated with $\begin{bmatrix}1\\0\end{bmatrix}$. In particular, unless $(i,j)$ is associated with $\begin{bmatrix}0\\0\end{bmatrix}$ or $\begin{bmatrix}1\\1\end{bmatrix}$, there is no occurrence of 0 above 1 in rows $i,j$ in $B$.

\subsection{Asymptotics of $\forb(m,3,p\cdot K_2)-\forb(m,3,p\cdot I_2)$}\label{sec:gp}
By \citet{dillon2021exponential}, if $2^{m-2}\geq p-1$, then $\forb(m,3,p\cdot K_2)=2^m+m2^{m-1}+(p-1)\binom m2$, but the exact value of $\forb(m,3,p\cdot I_2)$ is not known for general $p$. In this subsection, we determine asymptotically the difference between $\forb(m,3,p\cdot K_2)$ and $\forb(m,3,p\cdot I_2)$.

We first give a lower bound construction that uses only 0-mark and 1-mark columns.
\begin{lemma}\label{prelimlower}
$\forb(m,3,p\cdot I_2)\geq2^m+m2^{m-1}+(p-1)\left(\binom{m}{2}-\binom{r+1}{2}\right)+(r-1)2^r+1$.
\end{lemma}
\begin{proof}
Let every pair $i<j$ of rows be associated with $\begin{bmatrix}0\\1\end{bmatrix}$, so a mark in a column corresponds exactly to an instance of 0 above 1.     

If a column $v$ contains a unique mark at $(i,j)$, then $v_{\ell}\in\{1,2\}$ for all $1\le\ell<i$, $v_{\ell}=2$ for all $i<\ell<j$, and $v_{\ell}\in\{0,2\}$ for $j<\ell\le m$. Thus, there are exactly $2^{m-1+i-j}$ such columns, and we can include $\min\{p-1,2^{m-1+i-j}\}$ of them. Let $r=\ceil{\log_2{(p-1)}}$, then $2^{m-1+i-j}<p-1$ if and only if $i\in[r]$ and $m-r+i\leq j\leq m$. Therefore, the number of 1-mark columns we can include is
\begin{displaymath}(p-1)\left(\binom{m}{2}-\binom{r+1}{2}\right)+\sum_{i=1}^{r}\sum_{j=m-r+i}^{m}2^{m-1+i-j}=(p-1)\left(\binom{m}{2}-\binom{r+1}{2} \right)+(r-1)2^r+1.\end{displaymath}
Since there are $2^m+m2^{m-1}$ columns with no 0 above 1, and we can always include all of them, the result follows.
\end{proof}

If $A\in\ext(m,3,p\cdot I_2)$, let $A=\begin{bmatrix}B\mid C\end{bmatrix}$ be the standard decomposition. Since $|A|\geq2^m+m2^{m-1}+(p-1)\left(\binom{m}{2}-\binom{r+1}{2}\right)+(r-1)2^r+1$ by Lemma \ref{prelimlower}, and $|C|\leq(p-1)\binom{m}2$ from definition, we have that $|B|>2^m+m2^{m-1}-2^{m-3}$ as long as $m\geq 2r+2$. Hence, by Lemma~\ref{transitive}, we can suitably permute the rows of $A$ so that every pair $i<j$ of rows is associated with $\begin{bmatrix}0\\1\end{bmatrix}$. Note that in this case the directed graph $T$ is actually a tournament as there are no non-edges. 

Motivated by Lemma \ref{prelimlower} above,  we say a pair of rows $i<j$ is \textit{abundant} if there are at least $p-1$ possible columns with exactly one mark at $(i,j)$, or equivalently if $2^{m-1+i-j}\geq p-1$. A pair of rows $i<j$ is called \textit{scarce} if it is not abundant. It follows that if $r=\ceil{\log_2{(p-1)}}$, then there are $\binom{r+1}{2}$ scarce pairs of rows, which are exactly $(i,j)$ for $i\in[r]$ and $m-r+i\leq j\leq m$. The next lemma shows that we can always assume the maximum number of $1$-mark columns are present in $A$. 

\begin{lemma}\label{maxonemark}
There exists $A\in\ext(m,3,p\cdot I_2)$ containing $2^m+m2^{m-1}$ 0-mark columns and $(p-1)\left(\binom{m}{2}-\binom{r+1}{2}\right)+(r-1)2^r+1$ $1$-mark columns. More precisely, for every abundant pair $(i,j)$ of rows, $A$ contains exactly $p-1$ columns with their only mark at $(i,j)$, and for every scarce pair $(i,j)$ of rows, $A$ contains all columns with their only mark at $(i,j)$.
\end{lemma}
\begin{proof}
We can always take all $2^m+m2^{m-1}$ columns with no mark. Let $A\in\ext(m,3,p\cdot I_2)$ maximize the number of 1-mark columns. Assume for some pair $i<j$ of rows, $A$ does not contain $\min\{2^{m-1+i-j},p-1\}$ columns with their unique marks at $(i,j)$. Then there exists a column $v$ with its only mark at $(i,j)$ that is not in $A$. There are two cases. If $A$ contains less than $\min\{2^{m-1+i-j},p-1\}$ marks at $(i,j)$,  $\begin{bmatrix}A\mid v\end{bmatrix}$ is then in $\Avoid(m,3,p\cdot I_2)$ with more columns than $A$, contradiction. If $A$ contains exactly $\min\{2^{m-1+i-j},p-1\}$ marks at $(i,j)$, then $A$ contains a column $u$ with at least two marks including a mark at $(i,j)$. Remove column $u$ and add column $v$ to get a new matrix $A'$. Then $A'$ is in $\ext(m,3,p\cdot I_2)$ with one more $1$-mark column than $A$, contradiction. Thus, for all pairs $i<j$ of rows, $A$ contains $\min\{2^{m-1+i-j},p-1\}$ columns with their unique marks at $(i,j)$, which add up to $(p-1)\left(\binom{m}{2}-\binom{r+1}{2}\right)+(r-1)2^r+1$ $1$-mark columns in total. 
\end{proof}

We can now prove Proposition \ref{prop:gpindep}, which shows that $\forb(m,3,p\cdot K_2)-\forb(m,3,p\cdot I_2)$ is the same constant for all $m\geq 2r+2$.\\
\begin{proof}[Proposition \ref{prop:gpindep}]
Since $\forb(m,3,p\cdot K_2)=2^m+m2^{m-1}+(p-1)\binom{m}{2}$ by \citet{dillon2021exponential}, it suffices to show that for all $m\geq 2r+2$
\begin{align*}
    \forb(m+1,3,p\cdot I_2) - \forb(m,3,p\cdot I_2) &= \forb(m+1,3,p\cdot K_2) - \forb(m,3,p\cdot K_2)\\
    &= 2^{m-1}(m+4) + (p-1)m.
\end{align*}

First we prove that $\forb(m+1,3,p\cdot I_2) - \forb(m,3,p\cdot I_2) \geq 2^{m-1}(m+4) +(p-1)m$. Let $A\in\ext(m,3,p\cdot I_2)$. From above and by Lemma~\ref{transitive} we can assume $A$ contains at most $p-1$ instances of $0$ above $1$ in each pair $i<j$ of rows. $A$ must contain all $2^m + m2^{m-1}$ columns with no mark. Let $S$ denote the number of columns in $A$ with at least one mark. 

Consider the following matrix $A'$ with $m+1$ rows. First take all $2^{m+1} + (m+1)2^{m}$ columns with no $0$ above $1$. Then, fix an index $r+1\leq k\leq m + 1 - r$, which exists as $m\geq2r+2$. For every column $v$ in $A$ with at least one mark, define a new column $v'$ with $m+1$ rows by $v_i'=v_i$ if $i<k$, $v_k'=2$ and $v_i'=v_{i-1}$ if $k<i\leq m+1$. The new columns $v'$ obtained in this way are all distinct and have marks only on row set $\{1,2,\ldots ,k-1,k+1,\ldots m+1\}$ corresponding to the marks in the original columns. Add them all to $A'$. Moreover, for each $1 \leq j < k$, the pair $(j,k)$ is abundant so we can find $p-1$ columns with their unique marks at $(j,k)$. Similarly, for each $k < j\leq m+1$, the pair $(k,j)$ is abundant so we can find $p-1$ columns with their unique marks at $(k,j)$. Add all such $(p-1)m$ columns to $A'$. Then for any two rows $i<j$ of $A'$, there are at most $p-1$ columns with 0 above 1 by construction, so $A' \in \Avoid(m+1, 3, p \cdot I_2)$, which implies $\forb(m+1, 3, p\cdot I_2) \geq \abs{A'}  = 2^{m+1} + (m+1)2^{m} + S + (p-1)m$. Therefore,
\begin{align*}
    &\phantom{{}\geq{}}\forb(m+1, 3, p \cdot I_2) - \forb(m, 3, p \cdot I_2)\\
    &\geq \left( 2^{m+1} + (m+1)2^{m} + S + (p-1)m \right ) - \left (2^m + m2^{m-1} + S \right)\\
    &= 2^{m-1}(m+4) + (p-1)m.
\end{align*}

Now we prove that $\forb(m+1,3,p\cdot I_2) - \forb(m,3,p\cdot I_2) \leq 2^{m-1}(m+4) + (p-1)m$. Let $A\in\ext(m+1,3,p\cdot I_2)$. We can assume that $A$ has all $2^{m+1} + (m+1)2^{m}$ columns with no mark. By Lemma \ref{maxonemark}, we can also assume $A$ has the maximum number $(p-1)\left (\binom{m+1}{2}-\binom{r+1}{2} \right )+(r-1)2^r+1$ of $1$-mark columns. 

\indent Fix a row index $k$ with $r+1\leq k\leq m+1-r$. Let $S$ be the number of columns in $A$ with at least two marks and let $v$ be any one of them. For any abundant pair $(i,j)$, every column with a mark at $(i,j)$ is a $1$-mark column by Lemma \ref{maxonemark}, so all marks in $v$ correspond to scarce pairs. In particular, if $v$ has a mark at $(i,j)$, then $i\leq r<k<m+2-r\leq j$, which implies that $v_k=2$, as otherwise $v$ will have a mark at abundant pair $(i,k)$ or $(k,j)$. It follows that all columns in $A$ with at least two marks remain distinct and do not lose any mark after removing row $k$. Now we construct an $m$-rowed matrix $A'$ by taking all $2^m+m2^{m-1}$ columns with no mark, $(p-1)(\binom{m}{2}-\binom{r+1}{2})+(r-1)2^r+1$ columns with exactly one mark, as well as all columns obtained by removing the row $k$ entries from all columns in $A$ with at least two marks. Then $A'\in\Avoid(m,3,p\cdot I_2)$, so we have
\begin{align*}
\forb(m,3,p\cdot I_2)\geq\abs{A'}&=\left(2^m+m2^{m-1}\right)+(p-1)\left(\binom{m}{2}-\binom{r+1}{2}\right)+(r-1)2^r+1+S\\
&=\forb(m+1,3,p\cdot I_2)-2^{m-1}(m+4)-(p-1)m.
\end{align*}
This completes the proof.
\end{proof}

This allows us to define $g_p$ to be the common value of $\forb(m,3,p\cdot K_2)-\forb(m,3,p\cdot I_2)$ for all $m\geq 2r+2$. By \citet{dillon2021exponential}, we have $g_1=g_2=0, g_3=1, g_4=2$ and $g_5=5$. In the remaining part of this section, we prove Theorem \ref{thm:gpasymp}, which establishes the asymptotic formula $g_p\sim\frac12p(\log_2p)^2$. As will be seen later, it follows from Lemma \ref{prelimlower} that $g_p\leq\frac12p(\log_2p)^2+o(p(\log_2p)^2)$. To prove the reverse inequality, we use the following two results.

\begin{lemma}\label{foursigma}
Let $m\geq 2r+2$. Suppose $b,c\geq 1$ and $b+c\leq r+1$. Then there are at most $r2^r$ $m$-rowed $0,1,2$-columns with exactly $bc$ marks formed by $c$ 0's above $b$ 1's, all corresponding to scarce pairs of rows.
\end{lemma}
\begin{proof}
Let the positions of the $c$ 0's be $i_1<\cdots<i_c$ and the positions of the $b$ 1's be $j_1<\cdots<j_b$. All $bc$ marks correspond to scarce pairs if and only if $j_1-i_c\geq m-r$. 
The index $i_c$ must take values in $\{1,\ldots ,r\}$, so there are at most
$r$ possible choices of $i_c$. Once $i_c$ is fixed, the immediately following $(m - r - 1)$ positions are all equal to ‘2’, since all the marked
pairs of rows are scarce. This leaves $r$ positions still undetermined.
There are at most $2^r$ possibilities to choose which of these positions
are equal to ‘2’, and once this choice is made, the whole column
is uniquely determined, since the values of $b$ and $c$ are fixed. This
shows that there are at most $r2^r$ such columns.
\end{proof}
\begin{remark}
With tedious calculations a better upper bound  of such columns can be shown, namely there are at most $2^r(r-b-c+1)+2^{b+c-2}$  $0,1,2$-columns with exactly $bc$ marks formed by $c$ 0's above $b$ 1's, all corresponding to scarce pairs of rows.
\end{remark}
\begin{prop}\label{gplower}
Let $k\geq 2$ be fixed, then we have $g_p\geq\frac{(p-1)k}{k+1}\binom{r+1}{2}-r2^r\sum_{i=1}^k\tau(i)$, where $\tau(n)$ is the number of divisors of $n$.
\end{prop}
\begin{proof}
Let $A\in\ext(m,3,p\cdot I_2)$. As before we can assume $A$ contains at most $p-1$ instances of 0 above 1 in every pair of rows, and contains the maximum number of 1-mark columns. For $j\geq 0$, let $a_j$ be the number of $j$-mark columns in $A$. Note that if $j\geq 2$ and $v$ is a $j$-mark column, then every mark in $v$ corresponds to a scarce pair of rows. From before, we have $a_0=2^m+m2^{m-1}$ and $a_1=(p-1)\left(\binom{m}{2}-\binom{r+1}{2}\right)+(r-1)2^r+1$. Since there are at most $p-1$ marks at each pair of rows, we have $\sum_{i=1}^{\infty}ia_i\leq(p-1)\binom{m}{2}$. Let $\sum_{i=1}^{\infty}ia_i=(p-1)\binom{m}{2}-S$.

\indent Let $m\geq 2r+2$, $2\leq i\leq k$, and consider an $m$-rowed $0,1,2$-column $v$ with exactly $i$ marks, all of which are scarce. Since $m\geq 2r+2$, there does not exist $x<y<x'<y'$ such that both $(x,y)$ and $(x',y')$ are scarce. So there exists $b,c$ with $bc=i$ such that all $i$ marks of $v$ are formed as $c$ 0's above $b$ 1's. Thus Lemma \ref{foursigma} shows there are at most $\tau(i)r2^r$ such columns, and so $a_i\leq\tau(i)r2^r$. 

\indent We have 
\begin{displaymath}
\sum_{i=k+1}^{\infty}a_i\le \sum_{i=k+1}^{\infty}\frac{1}{k+1}ia_i=\frac{1}{k+1}\left((p-1)\binom{m}{2}-S-\sum_{i=1}^kia_i\right)
\end{displaymath}
and therefore \begin{displaymath}\sum_{i=1}^{\infty}a_i=\sum_{i=1}^{k}a_i+\sum_{i=k+1}^{\infty}a_i\leq \frac{1}{k+1}\left((p-1)\binom{m}{2}-S\right)+\sum_{i=1}^k\left(1-\frac{i}{k+1}\right)a_i.\end{displaymath}

\indent It follows that
\begin{align*}
g_p&=\forb(m,3,p\cdot K_2)-\forb(m,3,p\cdot I_2)=\forb(m,3,p\cdot K_2)-\abs{A}\\
&=(p-1)\binom{m}{2}-\sum_{i=1}^{\infty}a_i\\
&\geq\frac{(p-1)k}{k+1}\binom{r+1}{2}+\frac{S}{k+1}-\frac{k}{k+1}((r-1)2^r+1)-\sum_{i=2}^k\left(1-\frac{i}{k+1}\right)a_i\\
&\geq\frac{(p-1)k}{k+1}\binom{r+1}{2}-r2^r-\sum_{i=2}^ka_i\\
&\geq\frac{(p-1)k}{k+1}\binom{r+1}{2}-r2^r\sum_{i=1}^k\tau(i).\qedhere
\end{align*}
\end{proof}

\begin{proof}[of Theorem \ref{thm:gpasymp}]
By Lemma \ref{prelimlower}, $\forb(m,3,p\cdot I_2)\ge2^m+m2^{m-1}+(p-1)(\binom{m}{2}-\binom{r+1}2)+(r-1)2^r+1$. Hence, 
\begin{align*}
g_p&\leq\forb(m,3,p\cdot K_2)-2^m-m2^{m-1}-(p-1)\left(\binom{m}{2}-\binom{r+1}2\right)-(r-1)2^r-1\\
&=(p-1)\binom{r+1}2-(r-1)2^r-1=\frac12p(\log_2p)^2+o(p(\log_2p)^2).
\end{align*}

On the other hand, Proposition \ref{gplower} implies that for every fixed $k\geq2$, 
\begin{displaymath}g_p\geq\frac{(p-1)k}{k+1}\binom{r+1}{2}-r2^r\sum_{i=1}^k\tau(i)=\left(1-\frac1{k+1}\right)\frac12p(\log_2p)^2+o(p(\log_2p)^2).\end{displaymath}
Combining these two bounds and letting $k\to\infty$, we get $g_p\sim\frac12p(\log_2p)^2$.
\end{proof}

\subsection{$p\leq4$}\label{subsec:p34}
In this subsection, we compute the exact values of $\forb(m,3,F(a,b,c,d))$ when $\max\{a,d\}<\min\{b,c\}=p\leq 4$. If $p=1,2$, then $\forb(m,3,p\cdot I_2)=\forb(m,3,p\cdot K_2)$ by \citet{dillon2021exponential}. Then, for large $m$, Lemma \ref{reduction} implies that $\forb(m,3,p\cdot I_2)\leq\forb(m,3,F(a,b,c,d))=\forb(m,3,F(a,p,p,d))\leq\forb(m,3,p\cdot K_2)$, so $\forb(m,3,F(a,b,c,d))=\forb(m,3,p\cdot I_2)=\forb(m,3,p\cdot K_2)=2^m+m2^{m-1}+(p-1)\binom m2$.

Next, we treat the cases when $p=3$ and $p=4$.
\begin{lemma}\label{minp}
For $p\geq 3$, let $1\leq k\leq\min\{p,g_p\}$. If $\min\{b,c\}=p$, $\max\{a,d\}\leq p-k$, $m\geq 2r+2$ and $2^{m-2}\geq(\max\{b,c\}-1)m^2$, then \begin{displaymath}\forb(m,3,F(a,b,c,d))\leq\forb(m,3,p\cdot I_2)+g_p-k=\forb(m,3,p\cdot K_2)-k.\end{displaymath}
\end{lemma}
\begin{proof}
The conditions ensure that Lemma \ref{reduction} applies, so we can assume that $b=c=p$, and that $\forb(m,3,p\cdot I_2)=\forb(m,3,p\cdot K_2)-g_p$ by Proposition \ref{prop:gpindep}. Let $A\in\Avoid(m,3,F(a,p,p,d))$ and let $A=\begin{bmatrix}B\mid C\end{bmatrix}$ be a standard decomposition. Recall from the definition that $N$ is the number of pairs of rows that are associated with either $\begin{bmatrix}0\\0
\end{bmatrix}$ or $\begin{bmatrix}1\\1
\end{bmatrix}$. If $N=0$, then $A\in\Avoid(m,3,p\cdot I_2)$ so $\abs{A}\leq\forb(m,3,p\cdot I_2)=\forb(m,3,p\cdot K_2)-g_p$. If $N\geq 1$, then $\abs{A}\leq\abs{B}+\abs{C}\leq2^m+m2^{m-1}+(p-1)\binom{m}{2}-N(p-(p-k))\leq2^m+m2^{m-1}+(p-1)\binom{m}{2}-k=\forb(m,3,p\cdot K_2)-k$.
\end{proof}

\begin{cor}\label{p=3}
If $\min\{b,c\}=3$, $\max\{a,d\}\leq 2$, $m\geq 4$ and $2^{m-2}\geq(\max\{b,c\}-1)m^2$, then \begin{displaymath}\forb(m,3,F(a,b,c,d))=\forb(m,3,3\cdot I_2)=\forb(m,3,3\cdot K_2)-1.\end{displaymath}
\end{cor}
\begin{proof}
Set $p=3, k=1$ in Lemma \ref{minp} and note that $g_3=1$.
\end{proof}

\begin{prop}\label{p=4}
If $\min\{b,c\}=4$, $\max\{a,d\}\leq 3$, $m\geq 6$ and $2^{m-2}\geq(\max\{b,c\}-1)m^2$, then \begin{displaymath}\forb(m,3,F(a,b,c,d))=\forb(m,3,4\cdot I_2)=\forb(m,3,4\cdot K_2)-2.\end{displaymath}
\end{prop}
\begin{proof}
The conditions ensure that Lemma \ref{reduction} applies so we can assume that $b=c=4$. We have that $4\cdot I_2\prec F(a,4,4,d)$ and $g_4=2$ by \citet{dillon2021exponential}, so $2^m+m2^{m-1}+3\binom{m}{2}-2=\forb(m,3,4\cdot I_2)\leq\forb(m,3,F(a,4,4,d))$. It suffices to show the reverse inequality. 

\indent Let $A\in\Avoid(m,3,F(a,4,4,d))$, and let $A=\begin{bmatrix}B\mid C\end{bmatrix}$ be a standard decomposition. If $N=0$, then $A\in\Avoid(m,3,4\cdot I_2)$ so $\abs{A}\leq\forb(m,3,4\cdot I_2)$. If $N\geq 2$, then $\abs{A}\leq\abs{B}+\abs{C}\leq(2^m+m2^{m-1})+(3\binom{m}{2}-N)\leq\forb(m,3,4\cdot I_2)$.

\indent If $N=1$, $\abs{A}\leq\abs{B}+\abs{C}\leq(2^m+m2^{m-1})+(3\binom{m}{2}-1)=\forb(m,3,4\cdot I_2)+1$. Assume for a contradiction that equality holds. Then $\abs{B}=2^m+m2^{m-1}$, so $T$ is transitive by Lemma \ref{transitive}. Let $i'<j'$ be the unique pair of rows associated to $\begin{bmatrix}0\\0\end{bmatrix}$ or $\begin{bmatrix}1\\1\end{bmatrix}$. Without loss of generality we can assume it is associated with $\begin{bmatrix}0\\0\end{bmatrix}$. We can rearrange the rows of $A$ so that there is no 0 above 1 in $B$ for every pair of rows except possibly for rows $i',j'$. For equality to hold we also need $\abs{C}=3(\binom{m}{2}-1)+2=3\binom{m}{2}-1$. This means that every column in $C$ has exactly one mark, and for every $i<j$ with $(i,j)\ne(i',j')$ there are 3 columns in $C$ with a mark at $(i,j)$, and there are 2 columns in $C$ with a mark at $(i',j')$. There are four cases.

\indent If $1<i'<j'<m$, then for a column $v$ to have exactly one mark at $(1,m)$, we need $v_k=2$ for all $2\leq k\leq m-1$, which leaves only 1 choice for such $v$, contradiction. If $i'=1,j'=m$, then for a column $v$ to have exactly one mark at $(1,m-1)$, we need $v_k=2$ for all $2\leq k\leq m-2$. So there are only 2 choices for such $v$, which contradicts the above. If $i'=1,j'<m$, then for a column $v$ to have exactly one mark at $(1,m)$, we need $v_k=2$ for all $2\leq k\leq m-1$ except that $v_{j'}$ could be either $1$ or $2$, which gives only 2 choice for such $v$, contradiction. If $i'>1,j'=m$, then for a column $v$ to have exactly one mark at $(1,m)$, we need $v_k=2$ for all $2\leq k\leq m-1$ except that $v_{i'}$ could be either $0$ or $2$, which gives only 2 choice for such $v$, contradiction. In all cases we have a contradiction so equality cannot hold and $\abs{A}\leq\forb(m,3,4\cdot I_2)$. Thus $\forb(m,3,F(a,4,4,d))\leq\forb(m,3,4\cdot I_2)$, and this completes the proof.
\end{proof}

Up to now, we have shown that for $p\leq4$, $\forb(m,3,F(a,b,c,d))=\forb(m,3,p\cdot I_2)=\forb(m,3,p\cdot K_2)-g_p$ as long as $\max\{a,d\}<p=\min\{b,c\}$ and Lemma \ref{reduction} applies. This need not be true for larger $p$ as will be seen in the next subsection.

\subsection{$F(a,p,p,d)$}\label{sec:p-k}
Throughout this section we assume that $a,d\leq p$ and $A\in\ext(m,3,F(a,p,p,d))$.
\subsubsection{Structure of $A\in\ext(m,3,F(a,p,p,d))$}
Let $A=\begin{bmatrix}B\mid C\end{bmatrix}$ be a standard decomposition of $A$ as defined in Definition \ref{def:edge-nonedge}. Note that $|A|\geq \forb(m,3,p\cdot I_2)=\forb(m,3,p\cdot K_2)-g_p=2^m+m2^{m-1}+(p-1)\binom{m}2-g_p$ and $|C|\leq(p-1)\binom{m}{2}$. Thus, for all $m\geq 2N_r+2$,  we have $|B|>2^m+m2^{m-1}-2^{m-3}$, so by Lemma \ref{transitive} we may assume that no pair $i<j$ of rows is associated with $\begin{bmatrix}
    1\\0
\end{bmatrix}$. 
\begin{lemma}\label{lemma:transitive1}
Let $1\leq i<j<k\leq m$.
\begin{itemize}
    \item if $j,k$ forms a 0-non-edge, then one of $i,j$ and $i,k$ also forms a 0-non-edge,
    \item if $i,j$ forms a 1-non-edge, then one of $i,k$ and $j,k$ also forms a 1-non-edge.
\end{itemize} 
\end{lemma}
\begin{proof}
The proof is similar to the one of Lemma~\ref{transitive}. First assume that $i<j<k$ and $j,k$ forms a 0-non-edge, and neither $i,j$ nor $i,k$ form a 0-non-edge. For any subset $X\subset[m]$ of rows, let $B(X)$ be the matrix consisting of columns of $B$ with support $X$. Note that $B(X)|_X$ is simple and it is in $\Avoid(|X|,2,K_2)$, hence $|B(X)|_X|\le |X|+1$ if $\{i,j,k\}\not\subseteq X$. For $\{i,j,k\}\subseteq X$ we show below that $|B(X)|_X|\le |X|$, so then a similar calculation as in Lemma~\ref{transitive} shows $|B|\leq 2^m+m2^{m-1}-2^{m-3}$, a contradiction. 
\begin{description}
      \item[(i)] Pair $(i,j)$ is assigned $\begin{bmatrix}
          0\\1
      \end{bmatrix}$.
      \begin{description}
          \item[(ia)] Pair $(i,k)$ is assigned $\begin{bmatrix}
          0\\1
      \end{bmatrix}$. In this case entries in row $i$ of $B(X)$ must all be 1, since $i=0$ would imply both $j=0$ and $k=0$, which contradicts the assumption that $j,k$ forms a 0-non-edge. Therefore, $B(X)|_{X\setminus\{i\}}$ is simple.
      \item[(ib)] Pair $(i,k)$ is assigned $\begin{bmatrix}
          1\\1
      \end{bmatrix}$.  In this case the only  possible columns of $B(X)|_{\{i,j,k\}}$ are $\begin{array}{c|cc}
               i& 0&1\\j&0&1\\k&1&0
      \end{array}$, hence $B(X)|_{X\setminus\{i\}}$ is simple.
      \end{description}
      \item[(ii)]  Pair $(i,j)$ is assigned $\begin{bmatrix}
          1\\1
      \end{bmatrix}$.
      \begin{description}
      \item[(iia)]  Pair $(i,k)$ is assigned $\begin{bmatrix}
          0\\1
      \end{bmatrix}$. Now, the only  possible columns of $B(X)|_{\{i,j,k\}}$ are $\begin{array}{c|cc}
               i& 0&1\\j&1&0\\k&0&1
      \end{array}$, hence $B(X)|_{X\setminus\{i\}}$ is simple.
      \item[(iib)] Pair $(i,k)$ is assigned $\begin{bmatrix}
          1\\1
      \end{bmatrix}$. Here, row $i$ of $B(X)$ cannot contain any entry $1$, hence $B(X)|_{X\setminus\{i\}}$ is simple.
      \end{description}
\end{description}
In all cases above, it follows that $|B(X)|=|B(X)|_{X\setminus\{i\}}|\leq\forb(|X|-1,2,K_2)=|X|$, as required.

The case when $i<j<k$ and $i,j$ forms a 1-non-edge is similar. 
\end{proof}

\begin{lemma}\label{lemma:transitive2}
If $i<j<k$ and $i,k$ forms a non-edge, then one of $i,j$ and $j,k$ also forms a non-edge.
\end{lemma}
\begin{proof}
If $ij$ and $jk$ are both directed edges, then $ik$ would also be a directed edge by the transitivity provided by Lemma~\ref{transitive}, a contradiction.
\end{proof}
\begin{lemma}\label{lemma:nonedgelocation}
If there are $N$ non-edges, then all the 0-non-edges are in the first $N+1$ rows, and all the 1-non-edges are in the last $N+1$ rows.  
\end{lemma}
\begin{proof}
Let $k$ be the largest number such that there exists $j<k$ with $j,k$ forming a 0-non-edge. Then either $1,j$ or $1,k$ is also a 0-non-edge by Lemma~\ref{lemma:transitive1}. If $1,k$ is a 0-non-edge, then by  Lemma~\ref{lemma:transitive2}, for every $1<\ell<k$, either $1,\ell$ or $\ell,k$ is a non-edge. This means that the total number of non-edges is at least $k-1$. If instead $1,j$ is a 0-non-edge, then similarly for $1<i<j$, one of $1,i$ and $i,j$ is a non-edge, and for every $j<\ell<k$, one of $j,\ell$ and $\ell,k$ is a non-edge. Together with the non-edges $1,j$ and $j,k$ we get at least $(j-1)+(k-j)=k-1$ non-edges. In both cases, we conclude that $N\geq k-1$, so $k\leq N+1$. The 1-non-edges case is similar.
\end{proof}
\subsubsection{Lower bounds}
Similar to Section \ref{sec:gp}, we say a pair $(i,j)$ of rows with $i<j$ is \textit{abundant} if $2^{i-1+m-j}\geq p-1$, and \textit{scarce} otherwise. Equivalently, if $r=\ceil{\log_2(p-1)}$, then $(i,j)$ is abundant if $i-1+m-j\geq r$ and scarce otherwise. 

\begin{lemma}\label{lemma:correspondence}
Suppose $m\geq\max\{2N+2,2r+2\}$ and there are $N$ non-edges. Then 
\begin{itemize}
    \item for every scare pair $(i,m+1-j)$, if $a_i$ is the number of 0-non-edges $(i,k)$ with $i<k$, and $b_j$ is the number of 1-non-edges $(k,m+1-j)$ with $k<m+1-j$, then the number of columns with their unique marks at $(i,m+1-j)$ is $2^{i+j-2+a_i+b_j}$. 
    \item for every abundant pair $(i,j)$ that is not a non-edge, there are at least $p-1$ 1-mark columns, each of whose unique mark is at $(i,j)$.
    \item for every non-edge $(i,j)$, there are at least $p-1$ 1-mark columns, each of whose unique mark is at $(i,j)$.
\end{itemize}
\end{lemma}
\begin{proof}
Since $m\geq 2N+2$, the first and last $N+1$ rows of $A$ are disjoint. By Lemma \ref{lemma:nonedgelocation}, all 0-non-edges are within the first $N+1$ rows, and all 1-non-edges are within the last $N+1$ rows. 

If $(i,m+1-j)$ is a scarce pair, then $i-1+m-(m+1-j)\leq r-1$, so $(m+1-j)-i\geq m-r\geq\frac12m\geq N+1$. In particular, $(i,m+1-j)$ is not a non-edge. Let $i<k_1<k_2<\ldots <k_{a_i}$ be such that $(i,k_t)$ is a 0-non-edge for all $t\in[a_i]$, and let $\ell_1<\ell_2<\ldots <\ell_{b_j}<m+1-j$ be such that $(\ell_s,m+1-j)$ is a 1-non-edge for all $s\in[b_j]$. By Lemma \ref{lemma:nonedgelocation}, $k_{a_i}\leq N+1<m-N\leq\ell_1$. If $v$ is a 1-mark column with its unique mark at $(i,m+1-j)$, then we must have $v_i=0, v_{m+1-j}=1$ and $v_z=2$ for all $z\in\{i+1,i+2,\ldots ,m+1-j\}\setminus \{k_1,k_2,\ldots ,k_{a_i},\ell_1,\ell_2,\ldots ,\ell_{b_j}\}$, while $v_z$ can be 1 or 2 for all $z<i$ and all $z\in\{k_1,k_2,\ldots ,k_{a_i}\}$, and $v_z$ can be 0 or 2 for all $z>m+1-j$ and all $z\in\{\ell_1,\ell_2,\ldots ,\ell_{b_j}\}$. Moreover, every column $v$ of this form does contain exactly one mark at $(i,m+1-j)$, so there are $2^{i+j-2+a_i+b_j}$ 1-mark columns with their unique mark at $(i,m+1-j)$. 

If $(i,j)$ is abundant and not a non-edge, then $i-1+m-j\geq r$. It follows that we can find $0\leq r_1\leq i-1$ and $0\leq r_2\leq m-j$, such that $r_1+r_2=r$. Consider the columns $v$ of the form $v_i=0, v_j=1$, $v_z=1$ or 2 for all $z\leq r_1$, $v_z=0$ or 2 for all $z\geq m-r_2+1$, and $v_z=2$ for all other $z$. There are $2^{r_1+r_2}=2^r\geq p-1$ such columns, and one can verify that each of these contains a unique mark at $(i,j)$.

If $(i,j)$ is a non-edge, without loss of generality a 0-non-edge, consider the columns $v$ of the form $v_i=v_j=0$, $v_z=0$ or 2 for all $z\geq m-r+1$, and $v_z=2$ for all other $z$. There are $2^r\geq p-1$ such columns, and one can verify that each of these contains a unique mark at $(i,j)$.
\end{proof}

\begin{lemma}\label{lemma:0mark}
If there exist $r_1,r_2\geq0$ with $r_1+r_2\leq m-2$, such that the 0-non-edges are exactly $(i,j)$ with $1\leq i<j\leq r_1+1$ and the 1-non-edges are exactly $(i,j)$ with $m-r_2\leq i<j\leq m$, then there are exactly $2^m+m2^{m-1}$ columns with no mark. 
\end{lemma}
\begin{proof}
If $v_z=1$ or 2 for all $z\leq m-r_2-1$ and $v_z=0$ or 2 for all $z\geq m-r_2$, then $v$ contains no mark and there are $2^m$ such columns. 

If $v$ is not of this form and has no mark, then $v_z=0$ for some $z\leq m-r_2-1$, or $v_z=1$ for some $z\geq m-r_2$. In the first case, let $i\leq m-r_2-1$ be minimal such that $v_i=0$. If $i\leq r_1+1$, then $v_z$ must be 1 or 2 for all other $z\in[r_1+1]$ and $v_z$ must be 0 or 2 for all $z>r_1+1$, and all of these $2^{m-1}$ columns have no mark. If $r_1+1<i\leq m-r_2-1$, then $v_z$ must be 1 or 2 for all other $z<i$ and $v_z$ must be 0 or 2 for all $z>i$, and all of these $2^{m-1}$ columns have no mark. If we are not in the first case, then $v_z=1$ or 2 for all $z\leq m-r_2-1$. Let $i\geq m-r_2$ be minimal such that $v_i=1$. Then $v_z$ must be 0 or 2 for all other $z\geq m-r_2$, and all of these $2^{m-1}$ columns have no mark. In total, there are $m2^{m-1}$ such columns with no mark. 
\end{proof}

Recall that $N_r=\binom{\floor{\frac{r}{2}}+1}{2}+\binom{\ceil{\frac{r}{2}}+1}{2}$. We say that a scarce pair $(i,m+1-j)$ is \emph{fixed}, if $2^{i+j-2+a_i+b_j}\ge p-1$, or equivalently $i+j-2+a_i+b_j\geq r$, using the notation of Lemma~\ref{lemma:correspondence}.

\begin{proof}[of Proposition \ref{prop:lowerbound}]
Since $\forb(m,3,F(p-q_0,p,p,p-q_1))=\forb(m,3,F(p-q_1,p,p,p-q_0))$ by taking 0,1-complements, we may without loss of generality assume $q_0\leq q_1$. Consider $\binom{r_1+1}2$ 0-non-edges, which are $(i,j)$ for all $1\leq i<j\leq r_1+1$, and $\binom{r_2+1}2$ 1-non-edges, which are $(i,j)$ for all $m-r_2\leq i<j\leq m$. Using $r_1+r_2=r$, one can check that all $\binom{r+1}{2}$ scarce pairs are fixed. By Lemma \ref{lemma:0mark}, there are $2^m+m2^{m-1}$ columns with no mark. Since all scarce pairs are fixed, for each pair $(i,j)$ that is not a non-edge, by Lemma \ref{lemma:correspondence}, we can find $p-1$ columns, each of whose unique mark is at $(i,j)$. Also by Lemma \ref{lemma:correspondence}, for each 0-non-edge $(i,j)$, we can find $p-q_0-1$ columns with their unique marks at $(i,j)$, and for each 1-non-edge $(i,j)$, we can find $p-q_1-1$ columns with their unique marks at $(i,j)$. The matrix formed by all these distinct columns avoids $F(p-q_0,p,p,p-q_1)$ as a configuration, so
\begin{align*}
\forb(m,3,F(p-q_0,p,p,p-q_1))&\geq2^m+m2^{m-1}+(p-1)\binom m2-q_0\binom{r_1+1}2-q_1\binom{r_2+1}2\\
&=\forb(m,3,p\cdot K_2)-q_0\binom{r_1+1}2-q_1\binom{r_2+1}2,
\end{align*}
as required. 
\end{proof}

\begin{cor}\label{cor:lowerbound}
For every $0\leq q\leq p-1$ and every $m\geq 2N_r+2$,
\begin{displaymath}\forb(m,3,F(p-q,p,p,p-q))\geq\forb(m,3,p\cdot K_2)-qN_r.\end{displaymath}
\end{cor}
\begin{proof}
Apply Proposition \ref{prop:lowerbound} with $q_0=q_1=q$ and $r_1=\ceil{\frac r2}, r_2=\floor{\frac r2}$.
\end{proof}

Let us call a pair $(i,m+1-j)$ of rows \emph{weakly fixed} if, using the notation of Lemma \ref{lemma:correspondence}, we have $a_i+b_j=r+1-i-j$.

\begin{proof}[of Proposition \ref{prop:notalwaystrue}]
From assumption, $d=q-(\ceil{\frac r2}+1)(p-1-2^{r-1})>0$. We may assume $r\geq 2$ so $\floor{\frac r2}>0$, as otherwise the result follows directly from Corollary \ref{cor:lowerbound}. Consider $\binom{\ceil{\frac r2}+1}2$ 0-non-edges, which are $(i,j)$ for all $1\leq i<j\leq\ceil{\frac r2}+1$, and $\binom{\floor{\frac r2}}2=\binom{\floor{\frac r2}+1}2-\floor{\frac r2}$ 1-non-edges, which are $(i,j)$ for all $m-\floor{\frac r2}+1\leq i<j\leq m$. This corresponds to $r_1=\ceil{\frac r2}$ and $r_2=\floor{\frac r2}-1$ in the language of Proposition \ref{prop:lowerbound}, though as $r_1+r_2=r-1<r$, not every scarce pair is fixed. It can be checked that the scarce pairs $(i,m+1-j)$ are weakly fixed for all $i\in[\ceil{\frac r2}+1]$ and $j\in[\floor{\frac{r}2}]$, while all other scarce pairs are fixed. This means that we used $N_r-\floor{\frac r2}$ non-edges to fix $\binom{r+1}2-\floor{\frac r2}(\ceil{\frac{r}2}+1)$ scarce pairs, and leave the remaining $\floor{\frac r2}(\ceil{\frac{r}2}+1)$ scarce pairs weakly fixed. Therefore, $\forb(m,3,F(p-q,p,p,p-q))$ is at least the number of columns with 0 or 1 mark, which by Lemma \ref{lemma:correspondence} and Lemma \ref{lemma:0mark} is  
\begin{align*}
&\phantom{{}={}}2^m+m2^{m-1}+(p-1)\binom{m}{2}-q\left(N_r-\floor{\frac r2}\right)-\floor{\frac r2}\left(\ceil{\frac{r}2}+1\right)(p-1-2^{r-1})\\
&=\forb(m,3,p\cdot K_2)-qN_r+\floor{\frac r2}\left(q-\left(\ceil{\frac{r}2}+1\right)(p-1-2^{r-1})\right)\\
&=\forb(m,3,p\cdot K_2)-qN_r+\floor{\frac r2}d,
\end{align*}
as claimed.
\end{proof}
\subsubsection{Upper bounds}
\begin{lemma}\label{lemma:M-N}
If $m\geq 2N_r+2$, and $N$ non-edges are used to fix $M$ scarce pairs, then the maximum of $M-N$ taken ovar all matrices $A\in\ext(m,3, F (a, p, p, d))
$ is \begin{displaymath}\max (M-N)=\binom{r+1}2-N_r.\end{displaymath}
\end{lemma}
\begin{proof}
Since $M\leq\binom{r+1}2$, we may assume that $N\leq N_r$, so that $m\geq 2N+2$. For every $i\in [r]$, let $a_i$ be the number of 0-non-edges $(i,k)$ with $i<k$, and let $b_i$ be the number of 1-non-edges $(k,m+1-i)$ with $k<m+1-i$. By definition, for every $i\in[r]$ and $j\in[r+1-i]$, the scarce pair $(i,m+1-j)$ is fixed if and only if $2^{i+j-2+a_i+b_j}\geq p-1$, which happens if and only if $a_i+b_j\geq r+2-i-j$.

Consider the construction in the proof of Proposition~\ref{prop:lowerbound} with $r_1=\ceil{\frac r2}$ and $r_2=\floor{\frac r2}$. Here, $N_r=\binom{\ceil{\frac r2}+1}2+\binom{\floor{\frac r2}+1}2$ non-edges are used to fix all $\binom{r+1}2$ scarce pairs, so $M-N$ can achieve the value $\binom{r+1}2-N_r$. 

To show $\binom{r+1}2-N_r$ is the best possible, consider the following equivalent problem on the triangular array $T_r$ of size $r$, whose $(i,j)$ entry is $T_r(i,j)=r+2-i-j$ for all $i\in[r]$ and $j\in[r+1-i]$.
\begin{center}
$\begin{matrix}
r & r-1 & \cdots & \cdots & 1\\
\vdots & \vdots & \iddots & \iddots & \\
3 & 2 & 1\\
2 & 1\\
1\\
\end{matrix}$
\end{center}
For every $i\in [r]$, a row operation on row $i$ reduces every entry in row $i$ by 1, and a column operation on column $i$ reduces every entry in column $i$ by 1. If we perform $a_i$ row operations on row $i$ and $b_i$ column operations on column $i$ for all $i\in [r]$, then after all of these operations, the $(i,j)$ entry of $T_r$ is non-positive if and only if $a_i+b_j\geq r+2-i-j$, and so if and only if the scarce pair $(i,m+1-j)$ is fixed. Therefore, it suffices to show that if $N$ row or column operations turn $M$ entries in $T_r$ non-positive, then $M-N\leq\binom{r+1}2-N_r$.

To this end, we use induction on $r$ to show $M-N\leq\binom{r+1}2-N_r=\floor{\frac{r^2}{4}}$. The $r=1,2$ cases can be checked easily, so we assume $r\geq 3$. 

For each $i\in[r]$, let $a_i$ and $b_i$ be the number of operations performed on row $i$ and column $i$, respectively. Observe that we may assume $a_r=0$, as otherwise reducing $a_r$ by 1 and increasing $b_1$ by 1 does not decrease $M$ and leaves $N$ unchanged. Similarly, we may assume $b_r=0$.

Let $M_1$ be the number of non-positive integers in the first row or column of $T_r$ after $N$ operations, and let $M'=M-M_1$. Let $N'=\sum_{i=2}^{r-1}a_i+\sum_{i=2}^{r-1}b_j$, so that $N-N'=a_1+a_r+b_1+b_r=a_1+b_1$. From the induction hypothesis, $M'-N'\leq\floor{\frac{(r-2)^2}{4}}$, so we are done if $M_1-a_1-b_1\leq \floor{\frac{r^2}{4}}-\floor{\frac{(r-2)^2}{4}}=r-1$.

If the top left entry of $T_r$, which is $r$, is turned non-positive, then $a_1+b_1\geq r$, and so $M_1-a_1-b_1\leq2r-1-r=r-1$ and we are done.

If the top left entry of $T_r$ remains positive, let $x$ be the number of non-positive entries in row 1 and let $y$ be the number of non-positive entries in column 1, so that $x+y=M_1$. If $x\leq a_1$, then $y-b_1\leq y\leq r-1$, so $M_1-a_1-b_1=x-a_1+y-b_1\leq r-1$ and we are done. Similarly, we are done if $y\leq b_1$, so we assume neither is true. 

For each $2\leq i\leq r$, the $(1,i)$ entry of $T_r$ is non-positive after the operations if and only if $a_1+b_i\geq r+1-i$. Since $x$ of these entries are non-positive, we have $\sum_{i=2}^{r}b_i\geq\sum_{i=r-x+1}^r\max\{r+1-i-a_1,0\}=\sum_{i=1}^x\max\{i-a_1,0\}=\sum_{k=1}^{x-a_1}k=\binom{x-a_1+1}{2}$. Similarly, we have $\sum_{i=2}^{r}a_i\geq\binom{y-b_1+1}{2}$. Like above, we may assume $M_1-a_1-b_1=x-a_1+y-b_1\geq r$, as otherwise we are done. By convexity, $N\geq\sum_{i=2}^{r}a_i+\sum_{i=2}^{r}b_i\geq\binom{\floor{\frac{r+2}{2}}}2+\binom{\ceil{\frac{r+2}{2}}}2=N_r$, so $M-N\leq \binom{r+1}2-N_r=\floor{\frac{r^2}{4}}$, as required.
\end{proof}

\begin{lemma}\label{lemma:faraway}
If $m\geq 2N_r+2$ and $p\geq2^{r-1}+2q+1$, then \begin{displaymath}\forb(m,3,F(p-q,p,p,p-q))=\forb(m,3,p\cdot K_2)-qN_r.\end{displaymath}    
\end{lemma}
\begin{proof}
Suppose $N$ non-edges are used to fix $M$ scarce pairs. Let $\mathcal{P}$ be the set of scarce pairs $(i,m+1-j)$ that remained unfixed. For every $(i,m+1-j)\in\mathcal{P}$, let the $(i,j)$ entry of $T_r$ after the $N$ corresponding row or columns operations be $a_{i,j}$, so that by Lemma \ref{lemma:correspondence} there are $n_{i,j}:=2^{r-a_{i,j}}$ possible 1-mark columns with their unique marks at $(i,m+1-j)$. Since $(i,m+1-j)$ is not fixed, $a_{i,j}\geq1$ and so $p-1-n_{i,j}\geq p-1-2^{r-1}\geq2q$. 

It follows that the maximum number of 1-mark columns plus columns with at least two marks is at most
\begin{align*}
&\phantom{{}={}}(p-1)\left(\binom{m}{2}-\left(\binom{r+1}{2}-M\right)\right)-qN+\sum_{(i,m+1-j)\in\mathcal{P}}n_{i,j}+\frac12\sum_{(i,m+1-j)\in\mathcal{P}}(p-1-n_{i,j})\\
&=(p-1)\binom{m}{2}-qN-\frac12\sum_{(i,m+1-j)\in\mathcal{P}}(p-1-n_{i,j})\\
&\leq(p-1)\binom{m}{2}-qN-q\left(\binom{r+1}{2}-M\right)\\
&\leq(p-1)\binom{m}{2}-qN_r,
\end{align*}
where the last inequality follows from Lemma \ref{lemma:M-N}. 

There are at most $2^m+m2^{m-1}$ columns with no mark, as they together form a matrix in $\Avoid(m,3,K_2)$. Hence, the total number of columns is at most $2^m+m2^{m-1}+(p-1)\binom{m}{2}-qN_r=\forb(m,3,p\cdot K_2)-qN_r$. This combined with Corollary \ref{cor:lowerbound} completes the proof.
\end{proof}

\begin{lemma}\label{lemma:notfixed}
Let $m\geq 2N_r+2$, $qN_r\leq2^{r-3}$, and $A\in\Avoid(m,3,F(p-q,p,p,p-q))$. Suppose there is at least one scarce pair that is neither fixed nor weakly fixed, then $A$ has at most $\forb(m,3,p\cdot K_2)-qN_r$ columns.
\end{lemma}
\begin{proof}
Let $\mathcal{P}_1$ be the set of scarce pairs that are weakly fixed and let $\mathcal{P}_2$ be the set of scarce pairs that are neither fixed nor weakly fixed. For every $(i,m+1-j)\in\mathcal{P}:=\mathcal{P}_1\cup\mathcal{P}_2$, let the $(i,j)$ entry of $T_r$ after the $N$ operations be $a_{i,j}$, so $a_{i,j}=1$ if $(i,m+1-j)\in\mathcal{P}_1$, and $a_{i,j}\geq2$ if $(i,m+1-j)\in\mathcal{P}_2$. From Lemma \ref{lemma:correspondence}, there are $n_{i,j}:=2^{r-a_{i,j}}$ possible 1-mark columns with their unique marks at $(i,m+1-j)$.

Therefore, the maximum number of columns is at most, like in the proof of Lemma \ref{lemma:faraway},
\begin{align*}
&\phantom{{}={}}2^m+m2^{m-1}+(p-1)\binom{m}{2}-qN-\frac12\sum_{(i,j)\in\mathcal{P}}(p-1-n_{i,j})\\
&<2^m+m2^{m-1}+(p-1)\binom{m}{2}-qN-\frac12\cdot2^{r-2}\\
&\leq2^m+m2^{m-1}+(p-1)\binom{m}{2}-qN_r\\
&=\forb(m,3,p\cdot K_2)-qN_r,
\end{align*}
as required, where the first inequality uses $\emptyset\not=\mathcal{P}_2\subset\mathcal{P}$ and $p-1-n_{i,j}\geq p-1-2^{r-2}>2^{r-2}$ for all $(i,m+1-j)\in\mathcal{P}_2$, while the second inequality holds because $\frac12\cdot2^{r-2}\geq qN_r$.
\end{proof}

\begin{lemma}\label{lemma:LP}
If all scarce pairs are fixed or weakly fixed by $N$ non-edges, then $N_r-N\leq\ceil{\frac r2}$. 
\end{lemma}
\begin{proof}
In the language of Lemma \ref{lemma:M-N}, all entries in the triangular array $T_r$ are at most 1 after $N$ row or column operations, and we want to show $N\geq N_r-\ceil{\frac r2}$. We may assume that all operations are performed on the first $r-1$ rows and columns, as the entries in the last row and the last column are already at most 1. Now, apply this set of operations on the first $r-1$ rows and columns of $T_r$ to the triangular array $T_{r-1}$ instead, and note that they turn all $\binom r2$ entries in $T_{r-1}$ non-positive. Thus, by Lemma \ref{lemma:M-N} we have $\binom r2-N\leq\binom r2-N_{r-1}$, and so $N\geq N_{r-1}=N_r-\ceil{\frac r2}$.
\end{proof}

For $i\in[r]$ and $j\in[r+1-i]$, we say the $(i,j)$ entry of $T_r$ is fixed (or weakly fixed) if the corresponding scarce pair $(i,m+1-j)$ is fixed (or weakly fixed).
\begin{lemma}\label{lemma:induction}
If all scarce pairs are fixed or weakly fixed by $N_r-t$ non-edges for some $0\leq t\leq\ceil{\frac r2}$, then there are at least $t(\floor{\frac r2}+1)$ weakly fixed pairs. 
\end{lemma}
\begin{proof}
We use induction on $r$ and work equivalently in the triangular array $T_r$. The $r=1,2$ cases are easy to verify, so we assume $r\geq 3$.

For every $1\leq i\leq\ceil{\frac{r+1}2}$, let $a_i$ be the number of operations performed on row $i$ of $T_r$ and let $b_i$ be the number of operations performed on column $i$ of $T_r$. For every $1\leq i\leq\ceil{\frac{r+1}2}$, since the pair $(1,m+1-i)$ is fixed or weakly fixed, the $(1,i)$ entry of $T_r$, which was $r-i+1$ at the beginning, is now at most 1. Thus, $r-i+1-a_1-b_i\leq1$, and so $a_1+b_i\geq r-i$. In particular, $a_1+b_1\geq r-1$. Moreover, if the $(1,i)$ entry is non-positive, then $r-i+1-a_1-b_i\leq0$, and so $a_1+b_i\geq r-i+1$. Similarly, by looking at the $(i,1)$ entry of $T_r$ for every $1\leq i\leq\ceil{\frac{r+1}2}$, we see that $a_i+b_1\geq r-i$, and $a_i+b_1\geq r-i+1$ if this entry is non-positive.

Consider the subtriangle of size $r-2$ obtained by removing the first row and column from $T_r$. Let $N_{r-2}-S$ be the number of operations performed in this subtriangle. Then $r-1\leq a_1+b_1=N_r-t-(N_{r-2}-S)$, so $S\geq r-1-N_r+N_{r-2}+t=t-1$. Also, since all entries in this subtriangle are fixed or weakly fixed, we have $N_{r-2}-S\geq N_{r-2}-\ceil{\frac {r-2}2}=N_{r-2}-\ceil{\frac{r}2}+1$ by Lemma \ref{lemma:LP}. It follows that $S\leq\ceil{\frac r2}-1$, and we can apply the induction hypothesis to see that there are at least $S(\floor{\frac{r-2}2}+1)=S\floor{\frac r2}$ weakly fixed entries in this subtriangle.

If $S\geq t+1$, then as $S\leq\ceil{\frac r2}-1$ from above, we have $t\leq\ceil{\frac r2}-2$. From the induction hypothesis, the subtriangle contains at least $(t+1)\floor{\frac r2}\geq t(\floor{\frac r2}+1)$ weakly fixed entries, and hence so does the original triangle, as required. Otherwise, $S=t+x$ for $x=-1$ or 0. Then, we have $a_1+b_1=N_r-t-(N_{r-2}-S)=r+x$. Therefore, 
\begin{align*}
\sum_{i=2}^{\ceil{\frac {r+1}2}}(a_i+b_1)+\sum_{i=2}^{\ceil{\frac {r+1}2}}(a_1+b_i)&\leq\ceil{\frac {r-1}2}(r+x)+(N_{r-2}-S)\\
&=\begin{cases}
    \frac34r^2+\frac12(x-1)r-x-t, \text{ if }r\text{ is even},\\
    \frac34r^2+\frac12(x-2)r+\frac14(1-6x-4t), \text{ if }r\text{ is odd}.
\end{cases}
\end{align*}
On the other hand, \begin{displaymath}\sum_{i=2}^{\ceil{\frac{r+1}2}}(r-i+1)+\sum_{i=2}^{\ceil{\frac{r+1}2}}(r-i+1)=\begin{cases}
    \frac34r^2-\frac12r, \text{ if }r\text{ is even},\\
    \frac34r^2-r+\frac14, \text{ if }r\text{ is odd}.
\end{cases}\end{displaymath}

Therefore, if $r$ is even, then at least \begin{displaymath}\frac34r^2-\frac12r-\left(\frac34r^2+\frac12(x-1)r-x-t\right)=x+t-\frac12xr\end{displaymath}
of the inequalities $a_i+b_1\geq r-i+1$ for $2\leq i\leq\frac r2+1$ and $a_1+b_i\geq r-i+1$ for $2\leq i\leq\frac r2+1$ must fail to hold, which leads to $x+t-\frac12xr$ weakly fixed entries in the first row or column. If $x=-1$, this is $\frac12r+t-1$, which along with the weakly fixed $(1,1)$ entry and the at least $(t-1)\frac r2$ weakly fixed entries in the subtriangle guaranteed by the induction step above gives the required $t(\frac r2+1)$ weakly fixed entries overall. If $x=0$, this is $t$, and there are at least $t\frac r2$ weakly fixed entries in the subtriangle by induction, so we have the required $t(\frac r2+1)$ overall as well. 

If $r$ is odd, then at least \begin{displaymath}\frac34r^2-r+\frac14-\left(\frac34r^2+\frac12(x-2)r+\frac14(1-6x-4t)\right)=\frac32x+t-\frac12xr\end{displaymath}
of the inequalities $a_i+b_1\geq r-i+1$ for $2\leq i\leq\ceil{\frac{r+1}2}$ and $a_1+b_i\geq r-i+1$ for $2\leq i\leq\ceil{\frac{r+1}2}$ must fail to hold, leading to $\frac32x+t-\frac12xr$ weakly fixed entries in the first row or column. If $x=-1$, this is $\frac12r+t-\frac32=\floor{\frac r2}+t-1$, which along with the weakly fixed $(1,1)$ entry and the at least $(t-1)\floor{\frac r2}$ weakly fixed entries in the subtriangle guaranteed by induction gives the required $t(\floor{\frac r2}+1)$ weakly fixed entries overall. If $x=0$, this is $t$, and there are at least $t\floor{\frac r2}$ weakly fixed entries in the subtriangle by induction, so we have the required $t(\floor{\frac r2}+1)$ overall as well, finishing the proof. 
\end{proof}

\begin{proof}[of Theorem \ref{thm:p-k}]
We have already seen in Lemma \ref{lemma:faraway} that equality holds if $p\geq 2^{r-1}+2q+1$. If $p\leq2^{r-1}+2q$ and $qN_r\leq2^{r-3}$, then by Lemma \ref{lemma:notfixed}, the number of columns is at most $\forb(m,3,p\cdot K_2)-qN_r$ if at least one scarce pair is neither fixed nor weakly fixed. Now assume all scarce pairs fixed or weakly fixed and $\frac12(\floor{\frac r2}+1)(p-1-2^{r-1})\geq q$. 

Suppose there are $N=N_r-t$ non-edges, and they fix $M$ pairs, then by Lemma \ref{lemma:LP} we have $t\leq\ceil{\frac r2}$. Using the notation in the proof of Lemma \ref{lemma:notfixed}, the number of columns with at least 1 mark is at most
\begin{align*}
&\phantom{{}={}}(p-1)\binom{m}{2}-qN-\frac12\sum_{(i,m+1-j)\in\mathcal{P}_1}(p-1-n_{i,j})\\
&=(p-1)\binom{m}{2}-qN-\frac12|\mathcal{P}_1|(p-1-2^{r-1})\\
&=(p-1)\binom{m}{2}-qN_r+qt-\frac12|\mathcal{P}_1|(p-1-2^{r-1}).
\end{align*}

If $t\leq 0$, then this is at most $(p-1)\binom{m}{2}-qN_r$. If $1\leq t\leq\ceil{\frac r2}$, then Lemma \ref{lemma:induction} implies that $|\mathcal{P}_1|\geq t\left(\floor{\frac r2}+1\right)$, so this is at most
\begin{displaymath}(p-1)\binom{m}{2}-qN_r+qt-\frac12t\left(\floor{\frac r2}+1\right)(p-1-2^{r-1})\leq(p-1)\binom{m}{2}-qN_r.\end{displaymath}
Thus, in both cases, the maximum number of columns is at most
\begin{align*}
&\phantom{{}={}}2^m+m2^{m-1}+(p-1)\binom{m}{2}-qN_r=\forb(m,3,p\cdot K_2)-qN_r.
\end{align*}
Together with Corollary \ref{cor:lowerbound}, this completes the proof.
\end{proof}

\begin{proof}[of Corollary~\ref{thm:p-1ppp-1}]
If $p=2$, then $N_r=N_0=0$. From \citet{dillon2021exponential}, $g_2=0$, so we have $\forb(m,3,2\cdot K_2)=\forb(m,3,2\cdot I_2)$ for $m\geq 2$ by Proposition \ref{prop:gpindep}. Since $\forb(m,3,2\cdot I_2)\le \forb(m,3,F(1,2,2,1))\le \forb(m,3,2\cdot K_2)$, we have $\forb(m,3,F(1,2,2,1))=\forb(m,3,2\cdot K_2)=\forb(m,3,2\cdot K_2)-N_0$. If $p=3$, then $N_r=N_1=1$. By Corollary \ref{p=3}, noting that we do not need to apply the Lemma \ref{reduction}, for $m\geq4$ we have $\forb(m,3,F(2,3,3,2))=\forb(m,3,3\cdot K_2)-1=\forb(m,3,3\cdot K_2)-N_1$. If $p=4$, then $N_r=N_2=2$. By Proposition \ref{p=4}, again without the need for Lemma \ref{reduction}, for $m\geq6$ we have $\forb(m,3,F(3,4,4,3))=\forb(m,3,4\cdot K_2)-2=\forb(m,3,4\cdot K_2)-N_2$. Therefore, the statement holds for $p=2,3,4$.

For $p\ge 5$, we have $r\geq2$. It follows from Corollary \ref{cor:lowerbound} that $\forb(m,3, F(p-1,p,p,p-1))\geq\forb(m,3,p\cdot K_2)-N_r$, so we just need to prove the reverse inequality. If all scarce pairs are fixed or weakly fixed, then as $\frac12(\floor{\frac r2}+1)(p-1-2^{r-1})\geq1$, the proof of Theorem~\ref{thm:p-k} in the case when all scarce pairs are fixed or weakly fixed shows the maximum number of columns is $\forb(m,3,p\cdot K_2)-N_r$. 

Therefore, we just need to consider the case when at least one scarce pair is neither fixed nor weakly fixed. Assume that $N$ non-edges are used and $\mathcal{P}$ is the set of scarce pairs $(i,m+1-j)$ that are not fixed. Then, as in the proof of Lemma~\ref{lemma:notfixed}, the maximum possible number of columns is \begin{displaymath}2^m+m2^{m-1}+(p-1)\binom{m}{2}-N-\frac12\sum_{(i,m+1-j)\in\mathcal{P}}(p-1-n_{i,j}).\end{displaymath} This is at most $\forb(m,3,p\cdot K_2)-N_r$ if and only if 
\begin{equation}\label{eq:notfixed}
N+\frac12\sum_{(i,m+1-j)\in\mathcal{P}}(p-1-n_{i,j})\geq N_r.   
\end{equation}
Let $w_f$ denote the number of weakly fixed pairs and $n_f$ denote the number of neither fixed nor weakly fixed ones, so that $w_f+n_f=|\mathcal{P}|$. If $(i,m+1-j)$ is neither fixed nor weakly fixed, then $p-1-n_{i,j}\geq p-1-2^{r-2}\geq2^{r-2}+1$, so
\begin{equation}\label{eq:wfnf}
    N+\frac12\sum_{(i,j)\in\mathcal{P}}(p-1-n_{i,j})\geq N+\frac12 w_f+(2^{r-3}+\frac12)n_f.
\end{equation}
We will prove that for all $r$, 
\begin{equation}\label{eq:Nrupperbd}
    N+\frac12 w_f+(2^{r-3}+\frac12)n_f\geq N_r,
\end{equation}
which then implies (\ref{eq:notfixed}) and finishes the proof.

If $r\ge 7$, then $2^{r-3}\geq N_r$, and as $n_f\geq 1$, (\ref{eq:Nrupperbd}) holds. The values of $N_r$ for $2\leq r\le 6$ are as follows.
\begin{displaymath}\begin{array}{c|c|c|c|c|c|c|}
        r &2&3&4&5&6  \\
        \hline
         N_r&2&4&6&9&12 
\end{array}\end{displaymath}

If $r=2$, then we are done if $N\geq 1$ as $n_f\geq1$. If $N=0$, then $w_f=2$ and (\ref{eq:Nrupperbd}) holds as well. 

For $r=3$, we need $N+\frac12 w_f+\frac32n_f\geq 4$. If $N\ge3$, then the inequality holds as $n_f\geq1$. If $N=2$, (\ref{eq:Nrupperbd}) holds as at least one of the three diagonal entries is weakly fixed, while $N\le 1$ implies $n_f\ge 2$ and $w_f\geq 2$, so (\ref{eq:Nrupperbd}) still holds. 

For $r=4$, we need $N+\frac12 w_f+\frac52n_f\geq 6$, which holds if $N\geq 4$ as $n_f\geq1$. If $N=3$, it holds as $w_f\geq 1$. If $N=1,2$, then at least two of the three entries in the upper left corner of $T_4$ $\begin{array}{cc}4&3\\3&\end{array}$ are not at least weakly fixed, so $n_f\geq 2$ and (\ref{eq:Nrupperbd}) holds. If $N=0$, $n_f=6$ and (\ref{eq:Nrupperbd}) holds. 

For $r=5$, we need $N+\frac12 w_f+\frac92n_f\geq 9$, which holds if $N\geq 5$. If $N=4$, then (\ref{eq:Nrupperbd}) holds as $w_f\geq 1$. If $N\le 3$, then $n_f\ge 2$ and (\ref{eq:Nrupperbd}) holds, as at least two of the three upper left corner entries of $T_5$ $\begin{array}{cc}
         5&4  \\
         4& 
\end{array}$ are not at least weakly fixed. 

Finally for $r=6$, we need $N+\frac12 w_f+\frac{17}2n_f\geq 12$, which holds if $N\geq 4$. If $N\le 3$, then $n_f\geq3$ and (\ref{eq:Nrupperbd}) holds as all three entries of the upper left corner of $T_6$ $\begin{array}{cc}
         6&5  \\
         5& 
\end{array}$ are not at least weakly fixed. This proves (\ref{eq:Nrupperbd}) for all $r\geq 2$ and finishes the proof.
\end{proof}

\acknowledgements
The authors are indebted to the unknown referees for their careful reading of the paper. The authors are in particular thankful to one of the referees for significantly simplifying the proof of Lemma~\ref{foursigma}.

\end{document}